\documentclass[a4paper, 11pt, twoside, leqno]{article}

\usepackage[T1]{fontenc}
\usepackage[utf8x]{inputenc}
\usepackage[english]{babel}
\usepackage{lmodern}               
\usepackage{amsmath,amsthm,amssymb}
\usepackage{mathrsfs,esint,bbm,stmaryrd}
\usepackage{enumitem,graphicx,xcolor}
\usepackage[textwidth=16cm,centering,headheight=24pt]{geometry}
\usepackage{authblk}

\newtheorem{theorem}{Theorem}           
\newtheorem{corollary}[theorem]{Corollary}
\newtheorem{lemma}[theorem]{Lemma}
\newtheorem{prop}[theorem]{Proposition}

\theoremstyle{definition}              
\newtheorem{definition}{Definition}

\theoremstyle{remark}                  

\newtheorem{remark}[theorem]{Remark}

\DeclareMathOperator{\Id}{Id}                                       
\DeclareMathOperator{\dist}{dist}                                   
\DeclareMathOperator{\spt}{spt}                                     
\DeclareMathOperator{\curl}{\curl}                                  
\let\div\relax
\DeclareMathOperator{\div}{div}                                     

\newcommand{\abs}[1]{\left| #1 \right|}                             
\newcommand{\norm}[1]{\left\| #1 \right\|}                          
\newcommand{\mres}
{\mathbin{\vrule height 1.6ex depth 0pt width 0.13ex\vrule height 0.13ex depth 0pt width 1.3ex}}
\newcommand{\csubset}{\subset\!\subset}                             

\DeclareMathAlphabet{\mathpzc}{OT1}{pzc}{m}{it}

\newcommand{\J}{\mathrm{J}}
\renewcommand{\d}{\mathrm{d}}
\newcommand{\N}{\mathbb{N}}       
\newcommand{\R}{\mathbb{R}}
\newcommand{\Z}{\mathbb{Z}}

\newcommand{\F}{\mathbb{F}}

\renewcommand{\SS}{\mathbb{S}}
\newcommand{\G}{\mathbf{G}}       

\newcommand{\Q}{\mathbf{Q}}
\newcommand{\n}{\mathbf{n}}

\newcommand{\NN}{\mathscr{N}}     

\newcommand{\EE}{\mathscr{E}}

\renewcommand{\H}{\mathscr{H}}

\SetSymbolFont{stmry}{bold}{U}{stmry}{m}{n}  

\newcommand{\PR}{\mathbb{R}\mathrm{P}^2}

\renewcommand{\S}{\mathbf{S}}

\definecolor{lightblue}{rgb}{0.22,0.45,0.70}   
\definecolor{darkgray}{gray}{0.4}    
\definecolor{lightgray}{gray}{0.8}

\title{
 Improved partial regularity for manifold-constrained minimisers of subquadratic energies
}

\author{Giacomo Canevari and 
Giandomenico Orlandi\thanks{Dipartimento di Informatica --- Universit\`a di Verona,
Strada le Grazie 15, 37134 Verona, Italy. \\
\emph{E-mail addresses}: \texttt{giacomo.canevari@univr.it},
  \texttt{giandomenico.orlandi@univr.it}}}

\date{\today}

\begin{document}

\maketitle

\begin{abstract}
 We consider minimising $p$-harmonic maps from three-dimensional domains 
 to the real projective plane, for~$1<p<2$. These maps arise as least-energy 
 configurations in variational models for nematic liquid crystals. We show that
 the singular set of such a map decomposes into a $1$-dimensional
 set, which can be physically interpreted as a non-orientable line defect, and 
 a locally finite set, i.e. a collection of point defects.
%
\end{abstract}

\section{Introduction}
\label{sect:intro}

Nematic liquid crystals are an important example of an intermediate state of matter between liquids and crystalline solids.
The constituent molecules --- which, in many cases, have an elongated, rod-like shape ---
can flow and their centers of mass 
are randomly distributed in space, but the molecular axes tend to align locally to each other.
In other words, the molecules do not have positional order, but they have long-range orientational order. 
However, the orientational order is broken by the so-called defects, which are
regions of sharp contrast in the optical properties of the material. 
Nematic liquid crystals in a three-dimensional sample typically exhibit defects that are 
localised at isolated points or along lines.

From the point of view of continuum theories, 
the simplest way to describe a nematic liquid crystal is to introduce
a unit vector field, the molecular director~$\n$,
which represents the average orientation of the molecules at a given point.
However, in view of the (statistical) head-to-tail symmetry of the molecules,
it is more accurate to consider unoriented line fields or ``arrow-less vectors fields''
--- i.e., maps from the physical domain to the real projective plane $\PR$,
which is obtained by identifying antipodal pairs of points in the unit sphere~$\SS^2$.
In the liquid crystal literature, $\PR$ is usually embedded in a space of matrices, i.e.
$\PR$ is identified with the set of matrices~$\Q\in\R^{3\times 3}$ that can be written in the form
\begin{equation} \label{PR}
 \Q = \n\otimes\n - \frac{\mathbf{I}}{3} \qquad
 \textrm{for some } \n\in\SS^2,
\end{equation}
where~$(\n\otimes\n)_{ij} := n_in_j$ for~$i, \, j\in\{1, \, 2, \, 3\}$
and~$\mathbf{I}$ is the $(3\times 3)$-identity matrix.
A simple free energy functional, in non-dimensional form, is
\begin{equation} \label{quadr-en}
 I_2(\Q) := \frac{k}{2}\int_{\Omega} \abs{\nabla\Q}^2,
\end{equation}
where~$\Omega\subseteq\R^3$ is the domain, $k>0$ is a material-dependent parameter
and~$|\nabla\Q|^2:=\partial_k\ Q_{ij}\partial_k Q_{ij}$ is an elastic energy density
that penalises spacial inhomogeneities (Einstein's summation convention is assumed).
The corresponding space of finite-energy configurations is 
\begin{equation} \label{Sobolev}
 W^{1,2}(\Omega, \, \PR) := \left\{\Q\in W^{1,2}(\Omega, \, \R^{3\times 3})
 \colon \Q(x)\in\PR \textrm{ for a.e. } x\in\Omega\right\} \! .
\end{equation}
While this simple model is effective
in describing point defects~\cite{BrezisCoronLieb}, configurations that have a 
line discontinuity cost an infinite energy, and therefore are excluded from the theory.
Moreover, if~$\Omega$ is a simply connected, bounded, smooth domain then
any map~$\Q\in W^{1,2}(\Omega, \PR)$ can be \emph{oriented}; that is, there exists a vector
field~$\n$ \emph{as regular as~$\Q$} (i.e. $\n\in W^{1,2}(\Omega, \, \SS^2)$)
such that the representation~\eqref{PR} holds
pointwise a.e.~on~$\Omega$~\cite{BethuelChiron, BallZarnescu}. 
This result usefully allows to reformulate the problem in terms of unit vectors, 
instead of projective-valued maps; but, on the other hand, it implies 
that non-orientable configurations 
(such as the ``$+1/2$ or~$-1/2$ disclination lines'' which are frequently
observed in nematic liquid crystals, see Figure~\ref{fig:defects}) are not admissible.

\begin{figure}[t]
 \centering
 \includegraphics[height=.19\textheight]{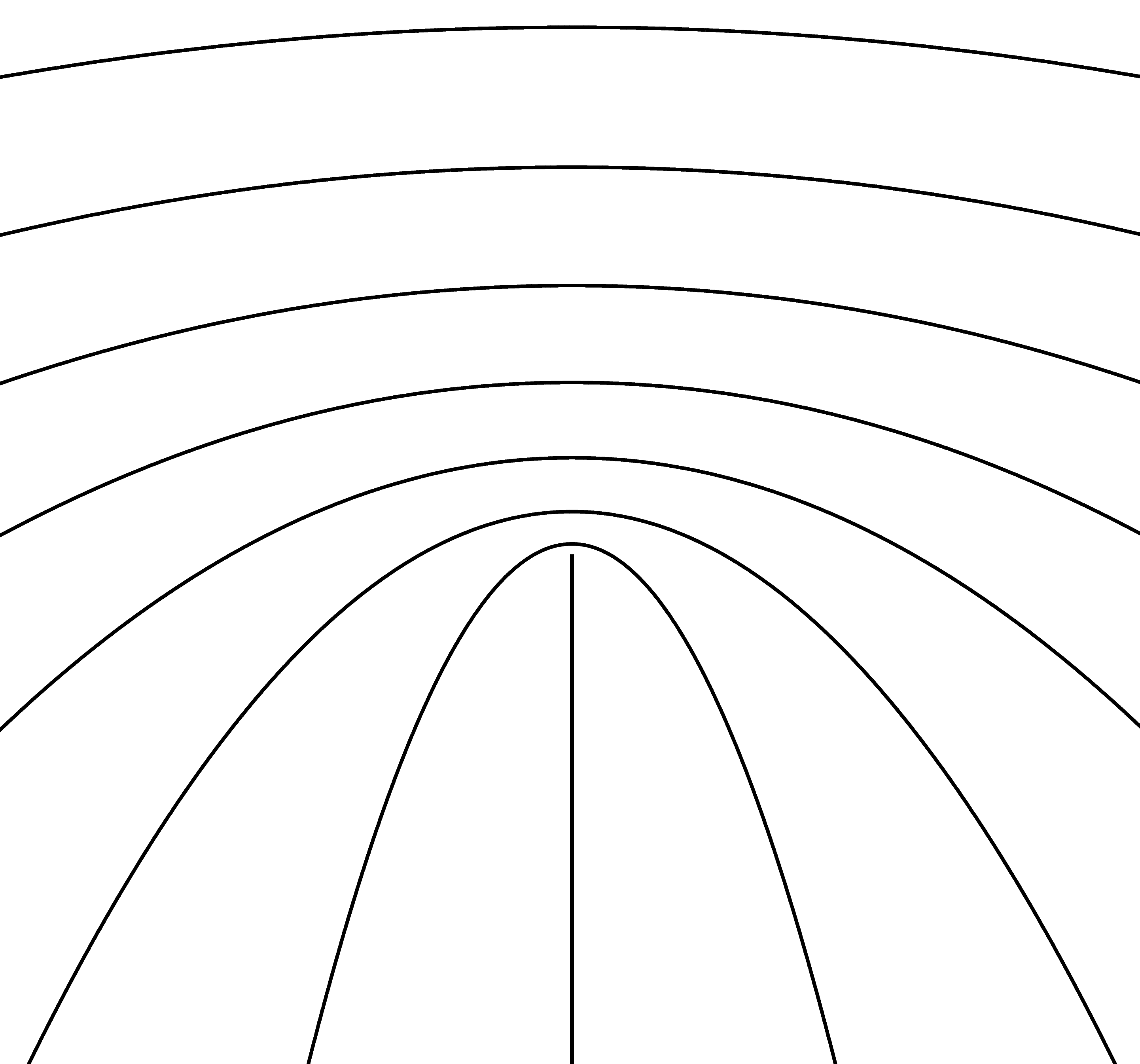}  \hspace{2.5cm}
 \includegraphics[height=.19\textheight]{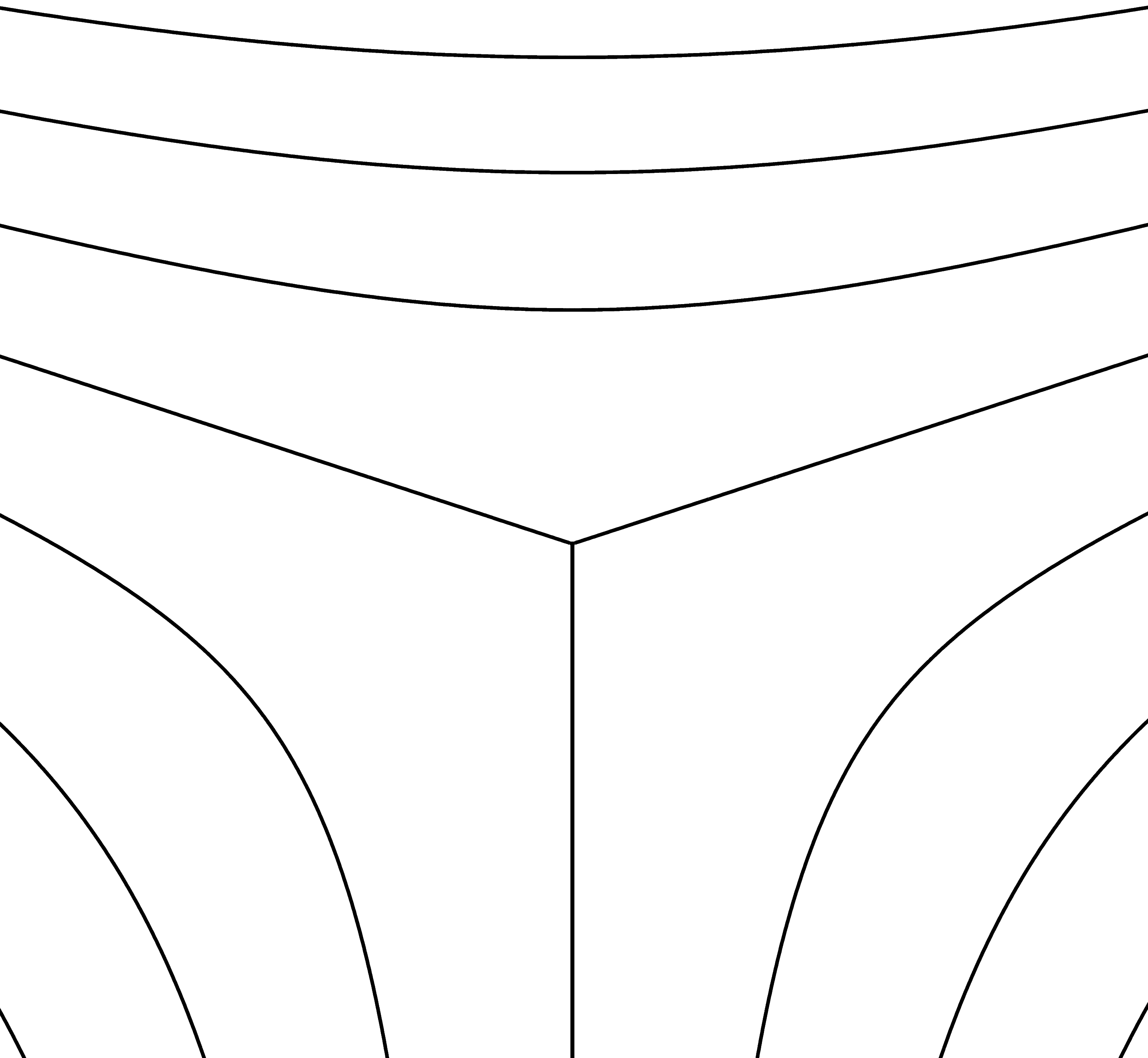}
 \caption{Non-orientable line defects, in cross section.
 The lines are tangent to the preferred orientation of the molecules, at each point.
 Left: a $+1/2$ disclination line; right: a $-1/2$ disclination line. }
 \label{fig:defects}
\end{figure}

A possible remedy for this deficiency is to replace the sharp constraint
that $\Q(x)\in\PR$ for a.e.~$x$ with a penalisation term in the free energy functional, 
i.e. to consider
\begin{equation} \label{LdG}
 I_{\mathrm{LdG}}(\Q) := \int_{\Omega} \left(
 \frac{1}{2}\abs{\nabla\Q}^2 + \frac{1}{\varepsilon^2} f_{\mathrm{LdG}}(\Q)\right)
\end{equation}
where~$\varepsilon>0$ is a (possibly small) parameter and~$f_{\mathrm{LdG}}$ is a
suitable potential that is minimised on~$\PR$.
The functional~\eqref{LdG} is known as the Landau-de Gennes model
(in the so-called one-constant approximation).
It has received great attention from the mathematical community in the last decade (see e.g. 
\cite{MajumdarZarnescu, Lamy-Hedgehog, DiFratta, INSZ-Hedgehog, INSZ-Instability, HenaoMajumdarPisante,
ContrerasLamy, AlamaBronsardLamy} and the references therein).
In the limit~$\varepsilon\to 0$, which is referred to as the limit of small 
nematic correlation length, one recovers the model~\eqref{quadr-en}, at least formally; 
rigorous statements can be found in~\cite{MajumdarZarnescu, GolovatyMontero, pirla, pirla3, ContrerasLamy-Convergence}.

An alternative approach is to perturb the elastic energy density.
While a quadratic elastic energy density is a reasonable modellistic 
assumption for small or moderate values of~$|\nabla\Q|$, there is no indication 
that the energy density should behave quadratically for large values of~$|\nabla\Q|$
and, indeed, functionals with sub-quadratic growth have been recently proposed as models
for nematic liquid crystals~\cite{BallBedford}. 
For illustration purposes, we consider here the functional
\begin{equation} \label{energy-p}
 I_p(\Q) := \frac{1}{p}\int_{\Omega} \abs{\nabla\Q}^p
 \qquad \textrm{for } 1 < p < 2,
\end{equation}
on the space~$W^{1,p}(\Omega, \, \PR)$ which is defined 
analogously to~\eqref{Sobolev}.
(Our results actually apply to a larger class of target manifolds and 
energy densities, see Section~\ref{sect:partial_regularity}.) Now 
configurations with line discontinuities, even non-orientable ones,
do belong to the admissible class. As in the quadratic case, the model~\eqref{energy-p}
can be obtained from a functional of the Landau-de Gennes type,
in the limit of small nematic correlation length \cite{GAB}.

Let us consider a minimiser~$\Q_0$ of~\eqref{energy-p} on a bounded,
smooth domain~$\Omega\subseteq\R^3$, subject to the boundary condition
\begin{equation} \label{BC}
 \Q = \Q_{\mathrm{bd}} \qquad 
 \textrm{on } \partial\Omega \textrm{ (in the sense of traces),}
\end{equation}
where~$\Q_{\mathrm{bd}}\in W^{1-1/p,p}(\partial\Omega, \, \PR)$ is given.
Thanks to \cite[Theorem~6.2]{HardtLin}, in case~$1 < p < 2$ there exist
maps~$\Q\in W^{1,p}(\Omega, \, \PR)$ that satisfy~\eqref{BC}; the existence of a minimiser
can then be proved by standard methods. The regularity theory for minimising $p$-harmonic 
maps~\cite{HardtLin, Luckhaus-PartialReg} implies that $\Q_0$ is locally
of class $C^{1, \alpha}$ out of the set
\begin{equation*}
 S(\Q_0) := \left\{x\in\Omega\colon \liminf_{\rho\to 0} \rho^{p-3}
 \int_{B_\rho(x)}|\nabla\Q_0|^p>0\right\} \! ,
\end{equation*}
and moreover, $S(\Q_0)$ is a relatively closed set of Hausdorff dimension at most~$1$.
Our main result provides additional information on the structure of the set~$S(\Q_0)$.
We denote by $\dim_{\H}(E)$ the Hausdorff dimension of a Borel subset $E\subseteq\R^3$.

\begin{theorem} \label{th:main}
 There exists a relatively closed set $\Sigma\subseteq S(\Q_0)$
 that satisfies the following properties:
 \begin{enumerate}[label=(\roman*)]
  \item for any ball~$B$ that intersects $\Sigma$, there holds
  $\dim_{\H}(\Sigma\cap B)=1$;
  \item for any ball~$B\csubset\Omega\setminus\Sigma$, the set
  $S(\Q_0)\cap B$ is finite.
 \end{enumerate}
\end{theorem}

In other words, the singular set~$S(\Q_0)$ splits into a $1$-dimensional 
part and a $0$-dimensional part, with the physical interpretation of 
being line defects and point defects, respectively. 
As it will appear from the proof, the map~$\Q_0$ has a 
\emph{non-orientable} singularity at~$\Sigma$: 
more precisely, if~$B\subseteq\Omega$ is an open ball that intersects~$\Sigma$,
then there is \emph{no} vector field~$\n\in W^{1,p}(B, \, \SS^2)$
such that~$\Q_0 = \n\otimes\n - \textbf{I}/3$ 
a.e. on~$B$ (see Definition~\ref{def:Sigmalift} below).
In particular, minimisers of~\eqref{energy-p} cannot have orientable line singularities.
This is consistent with the ``escape in the third dimension''
proposed by Cladis and Kl\'eman~\cite{CladisKleman}, which allows to 
de-singularise an orientable line defect via local surgery.
However, for non-minimising, stationary points of~\eqref{energy-p}
the situation is different, as they might have orientable line singularities as
well as non-orientable ones (see Remark~\ref{remark:stationary}).

Stratification results for the singular
set of minimising harmonic and $p$-harmonic maps are already 
available in the literature (see
e.g.~\cite{SchoenUhlenbeck, Simon-Harmonic, White-Stratification, CheegerNaber}).
In these results, the strata of the singular set are 
defined in terms of minimising tangent maps 
(that is, the blow ups of harmonic maps at a singular point)
and their regularity, or symmetry, properties.
The Hausdorff dimension of each stratum is bounded from above,
by means of Federer's ``dimension reduction'' argument. 
In Theorem~\ref{th:main}, by contrast, the set~$\Sigma$ 
is defined in terms of a topological property --- that is, 
the obstruction to local approximability by smooth maps.
This topological property, via White's results in 
geometric measure theory~\cite{White-Deformation}, implies
a \emph{lower} dimension bound on~$\Sigma$, which
is valid for any~$\Q\in W^{1,p}(\Omega, \, \PR)$ 
(see Theorem~\ref{th:partialdensity} below).
In the particular case~$\Q = \Q_0$ is a minimising $p$-harmonic map,
the topological lower bound on the dimension on~$\Sigma$
matches the upper bound given by the dimension reduction argument,
thus proving that~$\Sigma$ has Hausdorff dimension equal to one.

The key point in the proof of Theorem~\ref{th:main} is to approximate $\Q_0$
with smooth, $\PR$-valued maps. It is in general impossible to do so globally,
as smooth maps are \emph{not} sequentially dense in $W^{1,p}(\Omega, \, \PR)$
when~$1 < p < 2$, not even weakly, and not even when~$\Omega$ is a ball
(see e.g. \cite[Theorem~II]{BethuelZheng} and \cite[Theorem~3]{Bethuel-Density}).
Density of smooth maps is prevented by topological obstructions. 
However, to some extent it is possible to ``localise'' such topological obstructions;
that is, we can find a set~$\Sigma$ (whose dimension is bounded from below) such that, away from~$\Sigma$,
$\Q_0$ can locally be approximated by smooth maps.
This allows to define a local orientation for~$\Q_0$ away from~$\Sigma$ 
and hence, to reformulate locally the problem
in terms of unit vectors. 
This ``local density'' result (Theorem~\ref{th:partialdensity}) is 
obtained in Section~\ref{sect:local_density}. Similar results can be deduced
from an earlier work by Pakzad and Rivi\`ere~\cite{PakzadRiviere}, but the statement
we present here is more general and based on a different construction~\cite{CO1}.

The paper is organised as follows. Section~\ref{sect:local_density} is devoted to
the local approximability of manifold-valued maps with smooth maps. The main result of the section,
Theorem~\ref{th:partialdensity}, applies to a large class of target manifolds
and to a wider range of Sobolev exponents~$p$. We also discuss its application
to the lifting problem (Corollary~\ref{cor:partiallifting}). Section~\ref{sect:partial_regularity}
contains the proof of Theorem~\ref{th:main}. Again, the analysis is carried out in a 
broader setting than the one previously discussed; the main result of the section,
Theorem~\ref{th:stratification}, applies to target manifolds that are quotients 
of spheres and slightly more general energy functionals.

\numberwithin{equation}{section}
\numberwithin{definition}{section}
\numberwithin{theorem}{section}

\section{Local approximability by smooth maps}
\label{sect:local_density}

Let~$m\geq 2$, $k\geq 2$ be integers, and let~$\NN\subseteq\R^m$ 
be a smoothly embedded manifold without boundary.
We make the following assumption on~$\NN$ and~$k$:
\begin{enumerate}[label=(H), ref=\textrm{H}]
 \item \label{H} $\NN$ is compact and $(k-2)$-connected,
 that is $\pi_0(\NN) = \pi_{1}(\NN) = \ldots = \pi_{k-2}(\NN) = 0$. In case~$k=2$,
 we also assume that~$\pi_1(\NN)$ is abelian.
\end{enumerate}
For instance, this assumption is satisfied if~$\NN\simeq\PR$ (the case we considered 
in Section~\ref{sect:intro}) and~$k=2$,  because the fundamental
group of~$\PR$ is the group with two elements, $\pi_1(\PR)\simeq\Z/2\Z$.
The assumption~\eqref{H} has implications on the behaviour of 
$\NN$-valued maps and their singularities.
An arbitrary map~$\R^d\to\NN$ may be discontinuous 
on arbitrarily large sets. However, if the assumption~\eqref{H}
is in force, the singularities that occur on sets of
dimension~$d - k + 1$ or above carry no topological obstruction;
they can be ``smoothened out'' by local surgery. 
By contrast, singularities of dimension~$d-k$
may carry topological obstructions, which 
prevent approximability by smooth, $\NN$-valued maps.
This motivates the following

\begin{definition} \label{def:Sigma}
 Given a domain~$\Omega\subseteq\R^d$, a number~$p\in [1, \, +\infty)$
 and a map~$u\in W^{1,p}(\Omega, \, \NN)$, we define the set~$\Sigma(u)$
 as follows: $x\notin\Sigma(u)$ if and only if
 $x\notin\Omega$, or there exists
 an open ball~$B\subseteq\Omega$, $B\ni x$ such 
 that~$u_{|B}$ belongs to the strong~$W^{1,p}$-closure of $C^\infty(B, \, \NN)$.
\end{definition}

In other words, $\Sigma(u)$ is the set of points~$x\in\Omega$ such that
$u$ is \emph{not} approximable by smooth $\NN$-valued maps locally around~$x$,
in the strong~$W^{1,p}$-topology. The set~$\Sigma(u)$ is relatively closed in~$\Omega$,
as an immediate consequence of the definition.

\begin{theorem} \label{th:partialdensity}
 Let~$\NN$ be a smooth manifold, and let~$k\geq 2$ be an integer. We assume 
 that the condition~\eqref{H} is satisfied. Let~$\Omega\subseteq\R^d$ be a domain 
 of dimension~$d\geq k$. Let~$p\in [1, \, k)$ and~$u\in W^{1,p}(\Omega, \, \NN)$.
 Then, for any ball~$B\subseteq\Omega$, either~$\Sigma(u)\cap B = \emptyset$
 or~$\H^{d-k}(\Sigma(u)\cap B)>0$.
\end{theorem}

Condition~(i) implies that any connected component~$\Sigma^\prime$ of~$\Sigma(u)$
satisfies $\dim_{\H}\Sigma^\prime\geq d-k$. 
In general, there is no non-trivial upper bound on the dimension of~$\Sigma(u)$
and it may happen that~$\Sigma(u) = \Omega$ (see Remark \ref{remark:Riviere}).

When the Sobolev exponent satisfies~$1\leq p < k-1$, Theorem~\ref{th:partialdensity}
follows immediately from Bethuel's result~\cite{Bethuel-Density} 
and the assumption~\eqref{H}; in this case, we always have~$\Sigma(u) = \emptyset$.
Moreover, if $p$ satisfies the additional assumption $\lfloor p\rfloor\in\{1, \, d-1\}$,
Theorem~\ref{th:partialdensity} follows from earlier results by 
Pakzad and Rivi\`ere \cite[Theorem~II]{PakzadRiviere},
combined with White's deformation theorem \cite{White-Deformation}.
We are able to remove this technical restriction here 
because we appeal to the results in~\cite{CO1}, which are based on a different construction
than Pakzad and Rivi\`ere's one. However, the case~$p\geq k$ is not
covered by Theorem~\ref{th:partialdensity},
because the construction carried over in~\cite{CO1} breaks down in this case.

We also present an application of Theorem~\ref{th:partialdensity}
to the lifting problem. Let~$\pi\colon\EE\to\NN$ be the universal covering map of~$\NN$. 
We endow~$\EE$ with a pull-back Riemannian metric,  
in such a way that $\pi$ is a local isometry, and choose 
an isometric embedding of~$\EE$ 
as a closed subset of a Euclidean space~$\R^\ell$.
The manifold~$\EE$ need not be compact but it is smooth
and complete, and hence
such an embedding exists by Nash theorem~\cite{Nash56, Muller-ClosedEmbeddings}.
We say that $v\in W^{1,p}(\Omega, \, \EE)$
is a lifting of~$u$ if $u = \pi\circ v$ a.e. on~$\Omega$. In case~$\NN\simeq\PR$, 
we have~$\EE\simeq\SS^2$ and, identifying~$\NN$ with its embedding given by~\eqref{PR},
the map~$\pi\colon\SS^2\to\PR$ can be written as
\[
 \pi(\n) := \n\otimes\n - \frac{\mathbf{I}}{3} 
 \qquad \textrm{for any } \n\in\SS^2.
\]
Therefore, orienting a projective-valued map~$u$ amounts to finding a lifting of~$u$.
The existence of a (globally defined) lifting for maps $u\in W^{1, p}(\Omega, \, \NN)$
was studied by Bethuel and Chiron~\cite{BethuelChiron}
(and independently, in~\cite{BallZarnescu, Mucci-DCDS}
in case~$\NN=\R\mathrm{P}^2$). When~$p\geq 2$ and~$\Omega$ is simply connected, 
any~$u\in W^{1, p}(\Omega, \, \NN)$ has a lifting~$v\in W^{1, p}(\Omega, \, \EE)$;
however, for~$p<2$ there are examples of maps~$u\in W^{1,p}(\Omega, \, \NN)$
with no lifting in~$W^{1,p}$. The existence of liftings has been investigated
in different functional spaces, too.
The lifting problem in~$W^{s, p}(\Omega, \, \SS^1)$
was extensively studied by Bourgain, Brezis, and Mironescu,
see e.g.~\cite{BourgainBrezisMironescu, BourgainBrezisMironescu2005},
but results for more general targets~$\NN$ are also 
available~\cite{BethuelChiron, MironescuVanSchaftingen}.
The counterpart for Besov spaces has been considered in~\cite{MironescuRussSire}.
In the setting of BV-spaces, we refer to~\cite{GiaquintaModicaSoucek-II,DavilaIgnat,Ignat-Lifting} 
for the case $\NN=\SS^1$, 
to~\cite{IgnatLamy} for the case~$\NN=\R\mathrm{P}^q$,
and to~\cite{CO-Lifting} for general closed manifolds~$\NN$.

\begin{definition} \label{def:Sigmalift}
 Given a domain~$\Omega\subseteq\R^d$, a number~$p\in [1, \, +\infty)$
 and a map~$u\in W^{1,p}(\Omega, \, \NN)$, we define the set~$\Sigma_{\mathrm{lift}}(u)$
 as follows: $x\notin\Sigma_{\mathrm{lift}}(u)$ if and only if $x\notin\Omega$, 
 or there exists an open ball~$B\subseteq\Omega$, $B\ni x$,
 such that~$u_{|B}$ has a lifting~$v\in W^{1,p}(B, \, \EE)$.
\end{definition}

The definition immediately implies that
$\Sigma_{\mathrm{lift}}(u)$ is relatively closed in~$\Omega$.
If~$p\geq 2$, we have~$\Sigma_{\mathrm{lift}}(u) = \emptyset$ \cite{BethuelChiron}.

\begin{corollary} \label{cor:partiallifting}
 Let~$\NN\subseteq\R^m$ be a compact, smooth, connected manifold with abelian~$\pi_1(\NN)$.
 Let~$\Omega\subseteq\R^d$ be a domain of dimension~$d\geq 2$. 
 Let~$p\in [1, \, 2)$ and~$u\in W^{1,p}(\Omega, \, \NN)$. Then,
 \[
  \Sigma_{\mathrm{lift}}(u) = \Sigma(u).
 \]
 In particular, for any ball~$B\subseteq\Omega$, either
 $\Sigma_{\mathrm{lift}}(u)\cap B = \emptyset$ 
 or~$\H^{d-2}(\Sigma_{\mathrm{lift}}(u)\cap B)>0$.
\end{corollary}

The proof of Theorem~\ref{th:partialdensity} and Corollary~\ref{cor:partiallifting}
will be carried out in the setting of flat chains with coefficients 
in a normed abelian group~$(\G, \, |\cdot|)$. We follow here the notation adopted 
in~\cite[Section~2]{CO1}. In particular, we will use additive notation for the
group operation on~$\G$. We denote by $\F_n(\R^d; \, \G)$ 
(resp., $\F_n(\Omega; \, \G)$) the space of flat $n$-chains in~$\R^d$ 
(resp., in the domain $\Omega\subseteq\R^d$) with coefficients in the group~$\G$.
The support of a flat chain is denoted by~$\spt A$.
The restriction of a flat chain~$A$ to a Borel set~$E\subseteq\R^d$
will be denoted $A\mres E$. The reader is referred to \cite{Fleming, White-Rectifiability} 
for more details about flat chains. 

\begin{proof}[Proof of Theorem~\ref{th:partialdensity}]
 As noted above, the case~$p\in [1, \, k-1)$ is an immediate consequence of
 the results in~\cite{Bethuel-Density}, so we can assume w.l.o.g. that $k-1\leq p <k$.
 We can also assume that $\Omega$ is bounded and smooth (if not, we cover $\Omega$ 
 with countably many bounded, smooth subdomains). 
 As a consequence of the assumption~\eqref{H}, the group~$\pi_{k-1}(\NN)$ is abelian,
 finitely generated and can be endowed with a norm that satisfies
 \begin{equation*} 
  \inf_{g\in\G\setminus\{0\}} |g| > 0.
 \end{equation*}
 (see e.g. \cite[Section~2.2]{CO1}). 
 In~\cite[Theorem~1]{CO1}, the authors constructed a continuous operator
 \begin{equation} \label{Stop}
  \S^{\mathrm{PR}}\colon W^{1, p}(\Omega, \, \NN)\to \F_{d-k}(\Omega; \, \pi_{k-1}(\NN))
 \end{equation}
 with the property that, if $\Omega$ is a ball, then $\S^{\mathrm{PR}}(u) = 0$
 if and only if $u$ is a strong $W^{1,p}$-limit of smooth maps~$\Omega\to\NN$.
 Such an operator was previously constructed by Pakzad and Rivi\`ere 
 in case $\lfloor p\rfloor \in\{1, \, d-1\}$. 
 In particular, for any~$x\in\Omega$ we have
 \begin{equation} \label{S-Sigma}
  x\notin\Sigma(u) \quad \Longleftrightarrow \quad
  \textrm{there exists a ball } B\subseteq\Omega, B\ni x
  \textrm{ such that } \S^{\mathrm{PR}}(u_{|B}) = 0
 \end{equation}
 On the other hand, for any ball~$B\subseteq\Omega$,
 the chain~$\S^{\mathrm{PR}}(u_{|B})$ is indeed the
 restriction of $\S^{\mathrm{PR}}(u)$
 to~$B$ (in a suitable sense; see \cite[Corollary~3.3]{CO1}).
 Therefore, \eqref{S-Sigma} implies
 \begin{equation} \label{Sigma}
  \Sigma(u) = \spt\S^{\mathrm{PR}}(u).
 \end{equation}
 Now, if $B$ is a ball such that~$B\cap\spt\S^{\mathrm{PR}}(u)\neq\emptyset$
 then, possibly by taking a smaller ball~$B^\prime\csubset B$,
 the restriction $\S^{\mathrm{PR}}(u)\mres B^\prime$ is well-defined
 and~$\S^{\mathrm{PR}}(u)\mres B^\prime\neq 0$.
 (The restriction $S\mres B_r(x_0)$ need not be well-defined for any~$r>0$
 if $S$ has infinite mass, but it is still well-defined for a.e.~$r>0$;
 see e.g. \cite[Theorem~5.2.3.(2)]{DePauwHardt}).
 Then, \cite[Theorem~3.1]{White-Deformation} implies that
 $\H^{d-k}(\Sigma(u)\cap B^\prime) =
  \H^{d-k}(\spt\S^{\mathrm{PR}}(u)\cap B^\prime) > 0$,
 and the theorem follows.
\end{proof}

\begin{proof}[Proof of Corollary~\ref{cor:partiallifting}]
 Let~$p\in [1, \, 2)$, $u\in W^{1, p}(\Omega, \, \NN)$.
 We only need to prove that~$\Sigma_{\mathrm{lift}}(u) = \Sigma(u)$;
 then, the corollary will follow by applying Theorem~\ref{th:partialdensity}
 with the choice~$k=2$.
 Let~$B\subseteq\Omega$ be a ball, and let~$(u_j)_{j\in\N}$
 be a sequence of smooth maps~$B\to\NN$ that converge to~$u_{|B}$ strongly in~$W^{1,p}$.
 Standard topological results 
 (see e.g. \cite[Proposition~1.33]{Hatcher})
 imply that each~$u_j$ admits a
 smooth lifting $v_j\colon B\to\EE$. Moreover, the chain rule implies
 $|\nabla v_j| = |\nabla u_j|$ pointwise on~$B$, because~$\pi$ is a local isometry.
 Thus, up to subsequences, the $v_j$'s converge $W^{1,p}$-weakly and a.e.~to a 
 limit $v\in W^{1,p}(B, \, \NN)$, and~$v$ must be a lifting for~$u$.
 This shows~$\Sigma_{\mathrm{lift}}(u)\subseteq\Sigma(u)$.
 Conversely, suppose that $u_{|B}$ has a lifting~$v\in W^{1,p}(B, \, \EE)$.
 The universal covering~$\EE$ is simply connected and hence,
 $v$ is a strong~$W^{1,p}$-limit of smooth maps~$B\to\EE$.
 This follows from~\cite{Bethuel-Density} in case~$\EE$ is compact,
 and from~\cite[Corollaries~1.1 and~1.2]{BousquetPonceVanSchaftingen-Noncompact} 
 otherwise. By the chain rule and Lebesgue's dominated convergence theorem, 
 $u$ is also a strong $W^{1,p}$-limit of smooth maps~$B\to\NN$.
 Thus, $\Sigma_{\mathrm{lift}}(u)\supseteq\Sigma(u)$ and the proof is complete.
\end{proof}
 
\begin{remark} \label{remark:Riviere}
 As remarked above, it may happen that~$\Sigma(u)=\Omega$; this fact
 is exploited by Rivi\`ere~\cite{Riviere-harmonic_mostro} to construct
 a weakly harmonic map that is discontinuous on a dense set.
 Consider, for instance,
 the case $\Omega = B^3$ is the unit ball in~$\R^3$, $\NN$ is the unit sphere~$\SS^2\subseteq\R^3$,
 $k=3$ (the assumption~\eqref{H} is then satisfied, because~$\SS^2$ is simply connected) and~$p=2$.
 In this case, $\Sigma(u)$ is the support of the distributional Jacobian, which is 
 defined for any~$u\in (L^\infty\cap W^{1,2})(B^3, \, \R^3)$ as
 \begin{equation} \label{Jac}
   \J u = \frac{1}{3} \div (u\cdot\partial_2u\times\partial_3u, 
  \, u\cdot\partial_3u\times\partial_1u, \, u\cdot\partial_1u\times\partial_2u)
  \qquad \textrm{in the sense of } \mathscr{D}^\prime(B^3).
 \end{equation}
 Indeed, $u$ is a strong $W^{1,2}$-limit of smooth maps~$B^3\to\SS^2$
 if and only if~$\J u = 0$ on~$B^3$ \cite{Bethuel1990}.
 Let~$(Q_j)_{j\geq 1}$ be an enumeration of the points of $B^3$ with rational
 coordinates, and let $P_j := Q_j + (2^{-j}\pi, \, 0, \, \ldots, \, 0)$.
 Let~$B^\prime$ be a slightly larger ball, that contains all the $P_j$'s.
 By applying 
 \cite[Lemma~A2]{Riviere-harmonic_mostro} repeatedly, 
 we can construct a sequence of maps~$u_n\in H^1(B^\prime, \, \SS^2)$ such that
 $u_n$ is smooth on~$B^\prime\setminus\{P_j, \, Q_j\}_{j=1}^n$, 
 $\J u = \frac{4\pi}{3}\sum_{i=1} (\delta_{P_j} - \delta_{Q_j})$ and
 \[
   \norm{\nabla u_{n} - \nabla u_{n-1}}^2_{L^2(B^\prime)} 
   \leq C \abs{P_{n} - Q_{n}} \leq 2^{-n}\pi C,
 \]
 where~$C$ is an $n$-independent constant. In particular, 
 the sequence $(u_n)_{n\in\N}$ is Cauchy and it converges
 $W^{1,2}$-strongly to a limit~$u\in W^{1,2}(B^\prime, \, \SS^2)$.
 Since the distributional Jacobian is continuous with respect to the strong~$W^{1,2}$-convergence
 (this is readily checked using~\eqref{Jac}), we have $\J u = \frac{4\pi}{3}\sum_{j\geq 1}(P_j - Q_j)$
 and, noting that $Q_i \neq P_j$ for any~$(i, \, j)$, we conclude that
 $\Sigma(u) = \spt\J u \supseteq B^3$.
\end{remark}

\section{Partial regularity of minimising $\phi$-harmonic maps with values in quotients of spheres}
\label{sect:partial_regularity}

In this section, we give an application of Theorem~\ref{th:partialdensity}
to the partial regularity of minimisers for a manifold-constrained
variational problem, as anticipated in the introduction.
However, we consider a slightly more general setting than the one 
presented in Section~\ref{sect:intro}.

We consider the unit sphere~$\SS^q$ of dimension~$q\geq 2$.
Let~$\G$ be a finite, abelian group of isometries acting on~$\SS^q$.
We assume that the action is free, that is, $g(x)\neq x$ for any~$x\in\SS^q$ 
and any transformation~$g$ in~$\G$ other than the identity.
Then, the quotient space~$\NN := \SS^q/\G$ can naturally be given the structure 
of a smooth Riemannian manifold, with fundamental group $\pi_1(\NN)\simeq\G$
and universal covering~$\SS^q$ (see e.g.~\cite[Proposition~2.32]{Lee-Riemannian}, 
\cite[Proposition~1.40 p.~72]{Hatcher}).
The assumption~\eqref{H} is satisfied for~$k=2$.
By applying Nash isometric embedding theorem, we identify~$\NN$ with a submanifold 
of a Euclidean space~$\R^m$. If~$q=2$ and the group~$\G$ consists of the identity 
and of the map~$x\mapsto -x$, we have~$\NN=\PR$ and so we recover the case
presented in Section~\ref{sect:intro}.

\begin{remark} \label{remark:quotient}
 If~$q$ is even, then the only non-trivial group of isometries~$\G$ 
 that acts freely on~$\SS^q$ is~$\Z/2\Z$ and correspondingly,
 $\NN\simeq\R\mathrm{P}^q$. Indeed, $\G$ can be identified with a subgroup of the 
 orthogonal group~$\mathrm{O}(q+1)$, and the determinant defines a homomorphism 
 $\G\to\{1, \, -1\}$. If~$q$ is even, elements of $\mathrm{SO}(q+1)$
 have an eigenvalue~$1$, and because we have assumed that $\G$ acts freely
 on~$\SS^q$, we deduce that $\det(g) = -1$ for 
 any~$g\in\G\setminus\{\Id_{\SS^q}\}$. Therefore, for any 
 $g_1$, $g_2\in\G\setminus\{\Id_{\SS^q}\}$ we have $\det(g_1g_2) = 1$, hence
 $g_1g_2=\Id_{\SS^q}$. This implies that the group $\G$ consists of two elements only.
 However, there are other finite, abelian groups of isometries that act 
 freely on~$\SS^q$ if~$q$ is odd: for example, 
 the group generated by the isometry
 $(x_1, \, x_2, \, x_3, \, x_4, \ldots, \, x_{q}, \, x_{q+1})
 \mapsto (-x_2, \, x_1, \, -x_4, \, x_3, \, \ldots, -x_{q+1}, \, x_q)$, 
 which is cyclic of order four.
\end{remark}

Let~$1 < p < 2$, and let $\phi\colon[0, \, +\infty]\to [0, \, +\infty)$ be a
function (the elastic modulus) that satisfies the following assumptions:
\begin{enumerate}[label=(H\textsubscript{\arabic*}), ref=H\textsubscript{\arabic*}]
 \item \label{hp:C2} \label{hp:first} 
     $\phi\in C^1[0, \, +\infty)\cap C^2(0, \, +\infty)$;
 \item \label{hp:convex}
     $\phi(0)=\phi^\prime(0) = 0$ and $\phi^{\prime\prime}(t)>0$ for any~$t> 0$;
 \item \label{hp:subquadratic} \label{hp:last} the function defined by
     $\psi(t) := t\phi^\prime(t) - p\phi(t)$, for~$t\geq 0$, is bounded.
\end{enumerate}
An example of admissible~$\phi$ is defined by
\[
 \phi(t) := (t^2 + b)^{p/2} - b^{p/2} \qquad \textrm{for } t\geq 0,
\]
where~$b\geq 0$ is a fixed parameter. (In the applications to
liquid crystals modelling, one typically requires~$\phi(t)$ 
to behave quadratically for small~$t$, but this assumption is not needed
in the following analysis.)
As a consequence of \eqref{hp:first}--\eqref{hp:last}, the function
$\phi$ is strictly convex and strictly increasing, and in fact,
there exists a number~$\alpha> 0$ such that
\begin{equation} \label{blowup}
  \phi(t)  = \alpha t^p + \mathrm{O}(1), \quad
  \phi^\prime(t) = \alpha p \, t^{p-1} + \mathrm{O}(t^{-1}) \qquad 
  \textrm{as }t\to+\infty.
\end{equation}
(The proof of this claim, which is based on calculus only, is postponed.)

Now, let~$\Omega\subseteq\R^3$ be a bounded, smooth domain,
and let~$u_{\mathrm{bd}}\in W^{1-1/p, p}(\partial\Omega, \NN)$ be a boundary datum.
We consider the functional
\begin{equation} \label{energy}
 I_{\phi}(u) := \int_{\Omega} \phi(\abs{\nabla u}) ,
\end{equation}
to be minimised in the class
\begin{equation} \label{A}
 \mathscr{A} := \left\{u\in W^{1,p}(\Omega, \, \NN)\colon
 u = u_{\mathrm{bd}} \textrm{ on } \partial\Omega \textrm{ (in the sense of traces)} \right\} \! .
\end{equation}
The class~$\mathscr{A}$ is non-empty due to \cite[Theorem~6.2]{HardtLin},
and standard arguments imply the existence of minimisers for~$I_\phi$ in~$\mathscr{A}$.

Let~$u_0$ be a minimiser for~$I_\phi$ in~$\mathscr{A}$, and 
let~$S(u_0)\subseteq\Omega$ be defined by
\begin{equation} \label{singular}
 S(u_0) := \left\{x\in\Omega\colon \liminf_{\rho\to 0} \rho^{p-3}
 \int_{B_\rho(x)}\phi(|\nabla u_0|)>0\right\} \! .
\end{equation}
Under the assumptions \eqref{hp:first}--\eqref{hp:last}, the set~$S(u_0)$
is relatively closed in~$\Omega$, has Hausdorff dimension less than or equal to~$1$,
and moreover $u_0\in C^{1,\alpha}_{\mathrm{loc}}(\Omega\setminus S(u_0))$
for some~$\alpha\in (0, \, 1)$ (see, e.g., \cite{HardtLin} 
in case~$\phi(t)=t^p$, \cite{Luckhaus-Canberra, Luckhaus-PartialReg},
and \cite[Proposition~5.4]{diestroverJDE} for even more general energy 
densities that do not behave asymptotically as a power).
The main result of this section, which implies Theorem~\ref{th:main}, is the following:


\begin{theorem} \label{th:stratification}
 Suppose that the assumptions \eqref{hp:first}--\eqref{hp:last} hold.
 Then, there exists a relatively closed set $\Sigma\subseteq S(u_0)$
 such that
 \begin{enumerate}[label=(\roman*)]
  \item for any ball~$B$ that intersects $\Sigma$, there holds
  $\dim_{\H}(\Sigma\cap B)=1$;
  \item for any ball~$B\csubset\Omega\setminus\Sigma$, the set
  $S(u_0)\cap B$ is finite.
 \end{enumerate}
\end{theorem}

We now present the proof of Theorem~\ref{th:stratification}.
For the sake of completeness, we first give a proof of~\eqref{blowup}.

\begin{proof}[Proof of \eqref{blowup}]
 Set~$M := \sup_{t\geq 0}|t\phi^\prime(t) - p\phi(t)|$; $M$ is finite
 by assumption~\eqref{hp:subquadratic}. We have
 \[
   \abs{\frac{\d}{\d t}\left(\frac{\phi(t)}{t^p}\right)} \leq \frac{M}{t^{p+1}}
   \qquad \textrm{for any } t> 0
 \]
 and hence, by integrating,
 \begin{equation} \label{calculus1}
   \frac{\phi(s)}{s^p} - \frac{M}{p \, s^p} \leq \frac{\phi(t)}{t^p}
   \leq \frac{\phi(s)}{s^p} + \frac{M}{p \, s^p}
   \qquad \textrm{for any } t\geq s >0 .
 \end{equation}
 Therefore, the function $t\mapsto t^{-p}\phi(t)$ is bounded on~$[1, \, +\infty)$,
 and there must be a subsequence~$t_k\nearrow+\infty$ such that
 $t_k^{-p}\phi(t_k)$ converges to a finite limit~$\alpha\geq 0$.
 Taking~$t=t_k$ in~\eqref{calculus1}, and passing to the limit as~$k\to+\infty$,
 we deduce that $s^{-p}\phi(s) = \alpha + \mathrm{O}(s^{-p})$,
 i.e. $\phi(s) = \alpha s^p + \mathrm{O}(1)$, as~$s\to+\infty$.
 Note that~$\alpha$ cannot be zero, because $\phi$ is strictly
 increasing and strictly convex, hence $\phi(t)\to+\infty$ as~$t\to+\infty$.
 Finally, using~\eqref{hp:subquadratic} again, we obtain
 \[
  \phi^\prime(t) - \alpha p \, t^{p-1} = \frac{1}{t}\left(t \phi^\prime(t) - p\phi(t)\right)
  + \frac{p}{t}\left(\phi(t) - \alpha t^p\right) = \mathrm{O}(t^{-1}) 
  \qquad \textrm{as } t\to+\infty. \qedhere
 \]
\end{proof}

To prove Theorem~\ref{th:stratification}, we apply
Corollary~\ref{cor:partiallifting} so to reduce, locally out of an exceptional set,
to a $\SS^q$-valued problem instead of an $\NN$-valued one. We will then  
invoke the following property of 
minimising $\phi$-harmonic maps with values in the sphere~$\SS^q$.

\begin{prop} \label{prop:S2}
 Let~$B\subseteq\R^3$ be an open ball, let~$q\geq 2$ be an integer, let~$1 < p < 2$,
 and let $v_{\mathrm{bd}}\in W^{1-1/p, p}(\partial B, \, \SS^q)$.
 Let~$v_0$ be a minimiser of the functional~\eqref{energy} in the class
 \[
  \mathscr{B} := \left\{v\in W^{1,p}(B, \, \SS^q)\colon
  v=v_{\mathrm{bd}} \textrm{ on } \partial B \right\} \!,
 \]
 and let~$S(v_0)\subseteq B$ be defined as in~\eqref{singular}.
 Then, $S(v_0)$ is a locally finite set and 
 $v_0\in C^{1,\alpha}_{\mathrm{loc}}(\Omega\setminus S(v_0))$
 for some~$\alpha\in (0, \, 1)$.
\end{prop}

In turns, the proof of Proposition~\ref{prop:S2} is based on ``dimension reduction''
argument, and uses the following lemma (whose proof may be found,
e.g., in~\cite[Lemma~16]{GAB}).

\begin{lemma} \label{lemma:MTM}
 Let~$B^2\subseteq\R^2$ be the open, unit ball, let~$q\geq 2$
 be an integer, let~$1 < p < 2$. If a map $w_0\in W^{1,p}(B^2, \, \SS^q)$
 is $0$-homogenous (i.e. $w_0(x) = w_0(x/|x|)$ for a.e.~$x$)
 and satisfies
 \[
  \int_{B^2} \abs{\nabla w_0}^p \leq \int_{B^2} \abs{\nabla w}^p
  \quad \textrm{for any } w \in W^{1, p}(B^2, \, \SS^q)
  \textrm{ such that } w = w_0 \textrm{ on } \partial B^2,
 \]
 then~$w_0$ is constant.
\end{lemma}

\begin{remark} \label{remark:target}
 Theorem~\ref{th:stratification} can be generalised, with the same proof, 
 to compact target manifolds~$\NN$ with finite and abelian 
 fundamental group~$\pi_1(\NN)$, \emph{provided that} the analogue of Lemma~\ref{lemma:MTM}
 holds for maps $w_0$ with values in the universal covering space~$\EE$ of~$\NN$.
\end{remark}

\begin{remark} \label{remark:stationary}
 Lemma~\ref{lemma:MTM}
 fails for non-minimising critical points of~\eqref{energy}.
 Consider, for instance, the map $v\colon B^2\to\SS^2$ given by
 $v(x) := (x_1/\abs{x}, \, x_2/\abs{x}, \, 0)$ for $x = (x_1, \, x_2)\in B^2$.
 A computation shows that $v$ belongs to 
 $W^{1,p}(B^2, \, \SS^2)$ and is a stationary $p$-harmonic map, i.e.
 it satisfies the conditions
 \begin{gather*}
  -\partial_j \left(\abs{\nabla v}^{p-2}\partial_j v\right) 
     = \abs{\nabla v}^p v \quad \textrm{on } B^2\\
  -\partial_j \left(\partial_j v\cdot \partial_k v - 
     \frac{1}{p}\abs{\nabla v}^p\delta_{jk}\right) = 0
     \quad \textrm{on } B^2 \textrm{ for any } k\in\{1, \, 2\},
 \end{gather*}
 corresponding to criticality for~\eqref{energy} with respect 
 to outer and inner variations respectively, in case~$\phi(t) = t^p/p$.
 (Einstein's summation convention is assumed.)
 Moreover, if the denote by~$\pi$ the universal covering map~$\SS^2\to\PR$,
 then the map $\pi\circ v$ is a $\PR$-valued stationary $p$-harmonic map
 (because $\pi$ is a local isometry), and $\pi\circ v$ has an 
 \emph{orientable} singularity of codimension two  on the~$z$-axis.
\end{remark}

\begin{proof}[Proof of Proposition~\ref{prop:S2}]
 The proof of this result follows by classical arguments, e.g. the ones 
 in~\cite[Theorem~IV]{SchoenUhlenbeck} (for~$\phi(t) = t^2$) 
 or \cite[Theorem~4.5]{HardtLin} (in case~$\phi(t) = t^p$).
 We only sketch here the main steps in the proof and indicate how the
 arguments need be to adapted.
 
 The minimiser~$v_0$ satisfies an almost-monotonicity identity: namely,
 for any ball~$B_R(x_0)\csubset B$ and any $0<r<R$, there holds
 \begin{equation} \label{monotonicity}
  \begin{split}
   &R^{p-3}\int_{B_R(x_0)} \phi(\abs{\nabla v_0}) - r^{p-3}\int_{B_r(x_0)} \phi(\abs{\nabla v_0})
   + \int_r^R \xi^{p-4} \left(\int_{B_\xi(x_0)} \psi(\abs{\nabla v_0}) \right)\, \d\xi \\
   &\qquad\qquad = \int_{B_R(x_0)\setminus B_r(x_0)} \abs{x - x_0}^{p-3}
     \frac{\phi^\prime(\abs{\nabla v_0})}{\abs{\nabla v_0}} \abs{\partial_\nu v_0}^2 \, \d x \geq 0
  \end{split}
 \end{equation}
 where~$\nu := (x - x_0)/|x - x_0|$ and~$\psi$ is defined in~\eqref{hp:subquadratic}. 
 Since~$\psi$ is bounded by assumption, 
 the last term in the left-hand side is bounded by~$R^p - r^p$. 
 As usual, this identity is proved by considering the comparison 
 maps~$v_t := v_0\circ\varphi_t$, where~$(\varphi_t)_{t\in\R}$
 is the flow of a suitable radial vector field; see e.g. \cite[Lemma~5]{GAB} for more details.
 
 Now, let~$x_0\in B$ be fixed, and let~$v_\rho\colon B_1\to\SS^2$ be defined by
 $v_\rho(x) := v_0(x_0 + \rho x)$
 for~$x\in B_1$ and~$0<\rho<\dist(x_0, \partial B)$.
 Using~\eqref{blowup}, we see that
 \begin{equation*} 
  \rho^{p-3}\int_{B_\rho(x_0)} \phi(\abs{\nabla v_0}) 
  = \rho^p \int_{B_1} \phi\left(\frac{\abs{\nabla v_\rho}}{\rho}\right)
  = \alpha\int_{B_1} \abs{\nabla v_\rho}^p + \mathrm{O}(\rho^p) 
  \qquad \textrm{as }\rho\to 0,
 \end{equation*}
 and the left-hand side stays bounded as~$\rho\to 0$, due to~\eqref{monotonicity}.
 Thus, we can extract a subsequence~$\rho_j\searrow 0$ such that
 $v_{\rho_j}$ converges $W^{1,p}$-weakly to a map~$v_*\in W^{1,p}(B_1, \, \SS^q)$.
 Arguing as in~\cite[Theorem a]{Luckhaus-PartialReg} we deduce that, 
 for any~$\lambda\in (0, \, 1)$, ${v_*}_{|B_\lambda}$ is minimising $p$-harmonic 
 (subject to its own boundary conditions) and the convergence~$v_{\rho_j}\to v_*$
 is actually strong in~$W^{1,p}_{\mathrm{loc}}(B_\lambda)$.
 We take $0 <s\leq t < 1$ and apply the monotonicity formula~\eqref{monotonicity},
 with the choice~$R = t\rho_j$, $r=s\rho_j$.
 By changing variable in the integrals, we obtain
 \begin{equation*} 
  \begin{split}
   &\rho_j^{p-1}\int_{B_t\setminus B_s} \abs{x}^{p-3}
     \phi^\prime\left(\frac{|\nabla v_{\rho_j}|}{\rho_j}\right) 
     \frac{|\partial_\nu v_{\rho_j}|^2}{|\nabla v_{\rho_j}|}  \, \d x 
     - \rho_j^{p}\int_{s}^{t} \xi^{p-4} 
      \left(\int_{B_{\xi}} \psi\left(\frac{|\nabla v_{\rho_j}|}{\rho_j}\right) \right)\, \d\xi\\
   &\qquad\qquad= t^{p-3}\rho_j^p\int_{B_t} \phi\left(\frac{|\nabla v_{\rho_j}|}{\rho_j}\right)
     - s^{p-3}\rho^p_j\int_{B_s} \phi\left(\frac{|\nabla v_{\rho_j}|}{\rho_j}\right)
   \end{split}
 \end{equation*}
 Thanks to the strong convergence $v_{\rho_j}\to v_*$ 
 in~$W^{1,p}_{\mathrm{loc}}(B_\lambda)$, we have 
 $|\nabla v_{\rho_j}|/\rho_j\to +\infty$ a.e.
 (possibly up to a subsequence). Then, by applying~\eqref{blowup} and
 \eqref{hp:subquadratic}, we deduce that
 \begin{equation*} 
  \begin{split}
   p \int_{B_t\setminus B_s} \abs{x}^{p-d}
     |\nabla v_{\rho_j}|^{p-2}|\partial_\nu v_{\rho_j}|^2 \, \d x
   = t^{p-d}\int_{B_t} |\nabla v_{\rho_j}|^p
   - s^{p-d}\int_{B_s} |\nabla v_{\rho_j}|^p  + \mathrm{O}(\rho_j^p)
  \end{split}
 \end{equation*}
 and, passing to the limit as~$j\to+\infty$, we conclude that
 $\partial_\nu v_* = 0$ a.e., that is, $v_*$ is homogeneous of degree~$0$.
 At this point, we are in position to conclude the proof by repeating the arguments 
 in~\cite[Theorem~IV]{SchoenUhlenbeck}, with the help of Lemma~\ref{lemma:MTM}.
\end{proof}

\begin{proof}[Proof of Theorem~\ref{th:stratification}]
 Let~$\Sigma := \Sigma(u_0) = \Sigma_{\mathrm{lift}}(u_0)$.
 We have~$\Sigma\subseteq S(u_0)$
 because $u_0$ is continuous on $\Omega\setminus S(u_0)$ and,
 in particular, $u_0$ can be lifted locally away from~$S(u_0)$.
 By Corollary~\ref{cor:partiallifting} we know that $\dim_{\H}(\Sigma\cap B)\geq 1$
 for any ball~$B$ that intersects~$\Sigma$, while~$\dim_{\H} S(u_0)\leq 1$ by
 \cite{HardtLin, Luckhaus-PartialReg}.
 Therefore, we conclude that $\dim_{\H}(\Sigma\cap B)=1$ 
 for any ball~$B$ that intersects~$\Sigma$.
 
 Now, let~$B\csubset\Omega\setminus\Sigma$ be an open ball.
 By Corollary~\ref{cor:partiallifting},
 there exists a map~$v_0\in W^{1, p}(B, \, \SS^q)$ such
 that~$u_0 = \pi\circ v_0$ a.e. on~$B$,
 where $\pi\colon\SS^q\to\NN$ is the universal covering map.
 Since~$\pi$ is a local isometry (and hence $|\nabla v_0| = |\nabla u_0|$ a.e.),
 we deduce that $v_0$ is $\phi$-minimising harmonic
 in~$B$ subject to its own boundary condition,
 so $S(v_0)$ is locally finite by Proposition~\ref{prop:S2}.
 On the other hand, since $|\nabla v_0| = |\nabla u_0|$ a.e.,
 we have $S(u_0)\cap B = S(v_0)$ and the theorem follows.
\end{proof}

\section*{Acknowledgements}
 We acknowledge the anonymous referees for their careful
 reading of the manuscript and for their
 suggestions, which improved the presentation of the results.
 G. C.'s  research  was supported by the Basque Government
 through the BERC 2018-2021 program
 and by the Spanish Ministry of Economy and Competitiveness: 
 MTM2017-82184-R. G. O. was partially supported by GNAMPA-INdAM.


\bibliographystyle{plain}
\bibliography{singular_set}

\end{document}